\pdfoutput=1
\documentclass{amsart}\usepackage{amsmath}\usepackage{graphicx,color}\theoremstyle{definition}\theoremstyle{plain}\newtheorem{theorem}{Theorem}\newtheorem{lemma}{Lemma}\theoremstyle{remark}\newtheorem{remark}{Remark}
\begin{document}
\author{Noboru Ito}\address{Graduate School of Mathematical Sciences, ICMS, iBMath, The University of Tokyo, 3-8-1, Komaba Meguro-ku Tokyo, 153-8914, Japan}\email{noboru@ms.u-tokyo.ac.jp}
\author{Shosaku Matsuzaki}\address{Faculty of Education and Integrated Arts and Sciences, Waseda University, 1-6-1, Nishi Waseda, Shinjuku-ku, Tokyo, 169-8050, Japan}
\address{(Current address: Learning Support Center, Faculty of Engineering, Takushoku University, Tatemachi 815-1, Hachioji-shi, Tokyo, 193-0985, Japan)}
\email{smatsuza@ner.takushoku-u.ac.jp}\author{Kouki Taniyama}
\address{Department of Mathematics, School of Education, Waseda University, 1-6-1, Nishi Waseda, Shinjuku-ku, Tokyo, 169-8050, Japan}\email{taniyama@waseda.jp}\title[Circle arrangements of link projections]{Circle arrangements of link projections}\maketitle
\begin{abstract}
An oriented link projection is the image of a generic immersion of oriented circles into the 2-sphere. The circle arrangement of a link projection is a disjoint union of unoriented circles on the 2-sphere obtained by orientation-incoherent smoothing at each crossing point. We show that two oriented link projections have the same circle arrangement if and only if they are transformed into each other by certain local moves. We also show that every odd (resp.~even) component link has a link projection whose circle arrangement consists of exactly one (resp.~two) circles.\end{abstract}

		
\section{Introduction}
A {\it link} is a closed 1-manifold smoothly embedded into the 3-sphere $S^3$.
A {\it knot} is a link with exactly one component.
An oriented {\it link projection} is the image of a generic immersion of a disjoint union of oriented circles to the 2-sphere $S^2$.
A {\it diagram} of an oriented link, or an oriented {\it link diagram}, is an oriented link projection with over/under information of all double points.
A {\it crossing point} is a double point of a link projection or a link diagram.
For an arbitrary oriented link projection $P$, if the operation $A^{-1}$ shown in Figure~\ref{nosmoothing} (cf.~\cite{ItS}) is applied to every crossing point of $P$ and the orientation is forgotten, we have an isotopy class of unoriented simple closed curves on $S^{2}$.
The isotopy class is called a {\it circle arrangement} of $P$ (cf.~\cite{IT}) and denoted by $\tau(P)$. See Figure \ref{PtauP}.
For a diagram $D$ of an oriented link, $\tau(P_D)$ is denoted by $\tau(D)$, where $P_D$ is the oriented link projection obtained by forgetting over/under information of all crossing points of $D$.
		\begin{figure}[h!]\includegraphics[width=12.5cm]{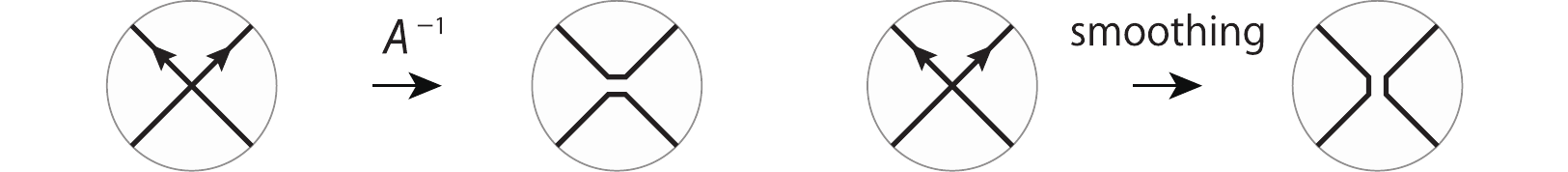}\caption{The operation $A^{-1}$ and the smoothing for a crossing point of an oriented link projection are depicted respectively.}\label{nosmoothing}\end{figure}
		
		\begin{figure}[h!]\includegraphics[width=12.5cm]{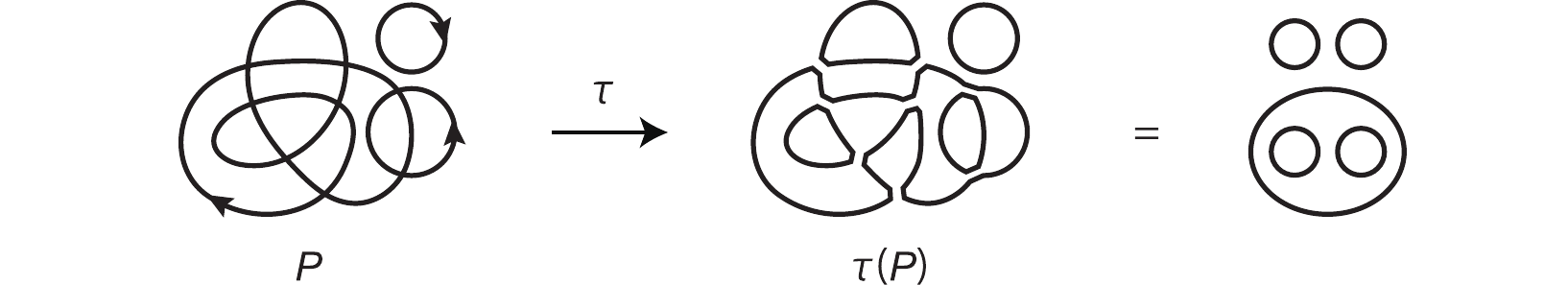}\caption{An oriented link projection $P$ and the circle arrangement $\tau(P)$.}\label{PtauP}\end{figure}
						
Let $P$ be an oriented link projection and let $P'$ be a connected component of $P$. We note that $P'$ may be an oriented link projection which is the image of a generic immersion of two or more circles. A simple closed curve $l$ contained in $P'$ is a {\it polygon} of $P$ if there exists a disk $D$ in $S^2$ such that ${\rm int}(D) \cap P' = \emptyset$ and $\partial D=l$. A polygon is an $i$-{\it gon} if the polygon contains exactly $i$ crossing points ($i \ge 0$).
A polygon $l$ of an oriented link projection is {\it coherent} if
all of the induced orientations on $l$ are mutually coherent.
We note that 1-gons and 0-gons are always coherent.

We define a local deformation {\it $M_i$-move}, denoted by $M_i$, as illustrated in Figure~\ref{Mmov} for each non-negative integer $i$. Each shaded area in Figure~\ref{Mmov} is an annulus on $S^2$. Every $M_i$ is a local move on an annulus, not on a disk as usual. Namely a link projection is unaltered outside the annulus by every $M_i$. Each $M_i$ has the two orientations ($i \ge 1$). Only one pattern of the orientation is illustrated in Figure \ref{Mmov}.\ 
Another one has inverse orientations.

		\begin{figure}[h!]\includegraphics[width=12.5cm]{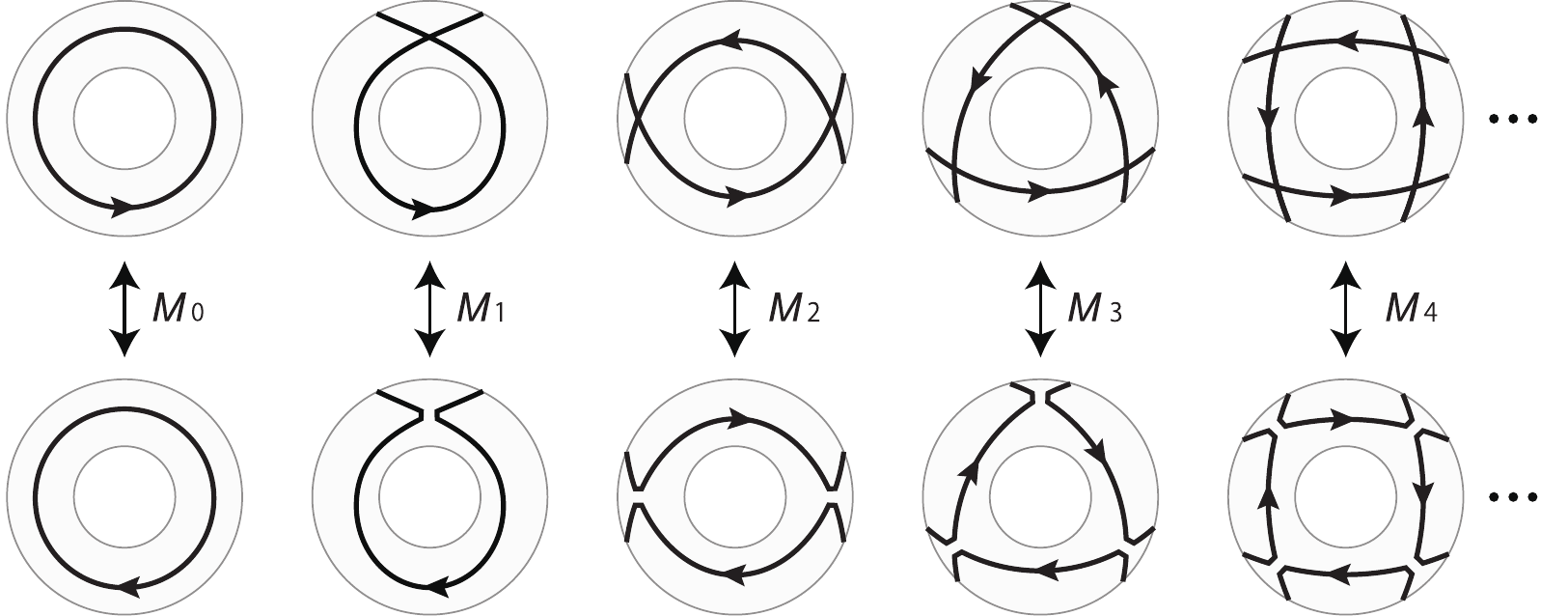}\caption{The $M_i$-move consists of the replacement depicted by $i$-gons ($i\ge0$).
Another type of $M_i$ with the inverse orientations is not illustrated here ($i \ge 1$).}\label{Mmov}\end{figure}
		
For a set $S$ of non-negative integers, oriented link projections $P$ and $Q$ are {\it S-equivalent}, denoted by $P \stackrel{S} {\sim} Q$, if there exists a finite sequence $P_0, P_1, P_2, \ldots, P_m$ of oriented link projections such that $P_0=P$, $P_m=Q$ and for every $k$ with $0 \le k \le m-1$ the oriented link projection $P_k$ is deformed to $P_{k+1}$ by an $M_i$-move with $i \in S$.
By applying operations, each of which is $A^{-1}$, to the diagrams of the top row in Figure~\ref{Mmov}, we have the diagrams of the bottom row. Therefore $P \stackrel{\{ i \}} {\sim} Q$ implies $\tau(P)=\tau(Q)$ for each non-negative integer $i$. 

\begin{theorem}\label{mainthm1}Let $n$ be an integer with $n \ge 3$. For oriented link projections $P$ and $Q$, the following conditions are equivalent.
 \[
 (1) \ \tau(P)=\tau(Q). \ \ \ \ \ \ (2) \ P \stackrel{ \mathbb{Z}_{\ge 0}} {\sim} Q.  \ \ \ \ \ \ (3) \ P \stackrel{ \{ 1, \ n \}} {\sim}  Q.
 \]
 \end{theorem}

The number of circles contained in the circle arrangement of an oriented link projection $P$ (resp.~link diagram $D$) is denoted by $\#\tau(P)$ (resp.~$\#\tau(D)$).

\begin{theorem}\label{mainthm2}
Let $L$ be an oriented link which has exactly $n$ components. If $n$ is odd $($resp.~even$)$, then the minimum of $\# \tau(D)$ where $D$ varies over all diagrams of $L$ is one $($resp.~two$)$.\end{theorem}
		
		
\section{Proof of Theorem~\ref{mainthm1}}\label{sec2}

		\begin{lemma}\label{lemmaMi}
For any integers $n \ge 3$ and $i \ge 0$, $M_i$ is obtained from $M_1$ and $M_n$.\end{lemma}
		
		\begin{proof}We consider the following two cases.

\vskip 2mm
\noindent
{\bf Case 1.} $0 \le i \le n$.

\noindent
For every non-negative integer $k$, $M_{k}$ is obtained from $M_1$ and $M_{k+1}$ by using the process shown in Figure \ref{Mdown}. Therefore, the claim holds for $i=n, n-1, \ldots, 1, 0$.

\begin{figure}[h!]\includegraphics[width=12.5cm]{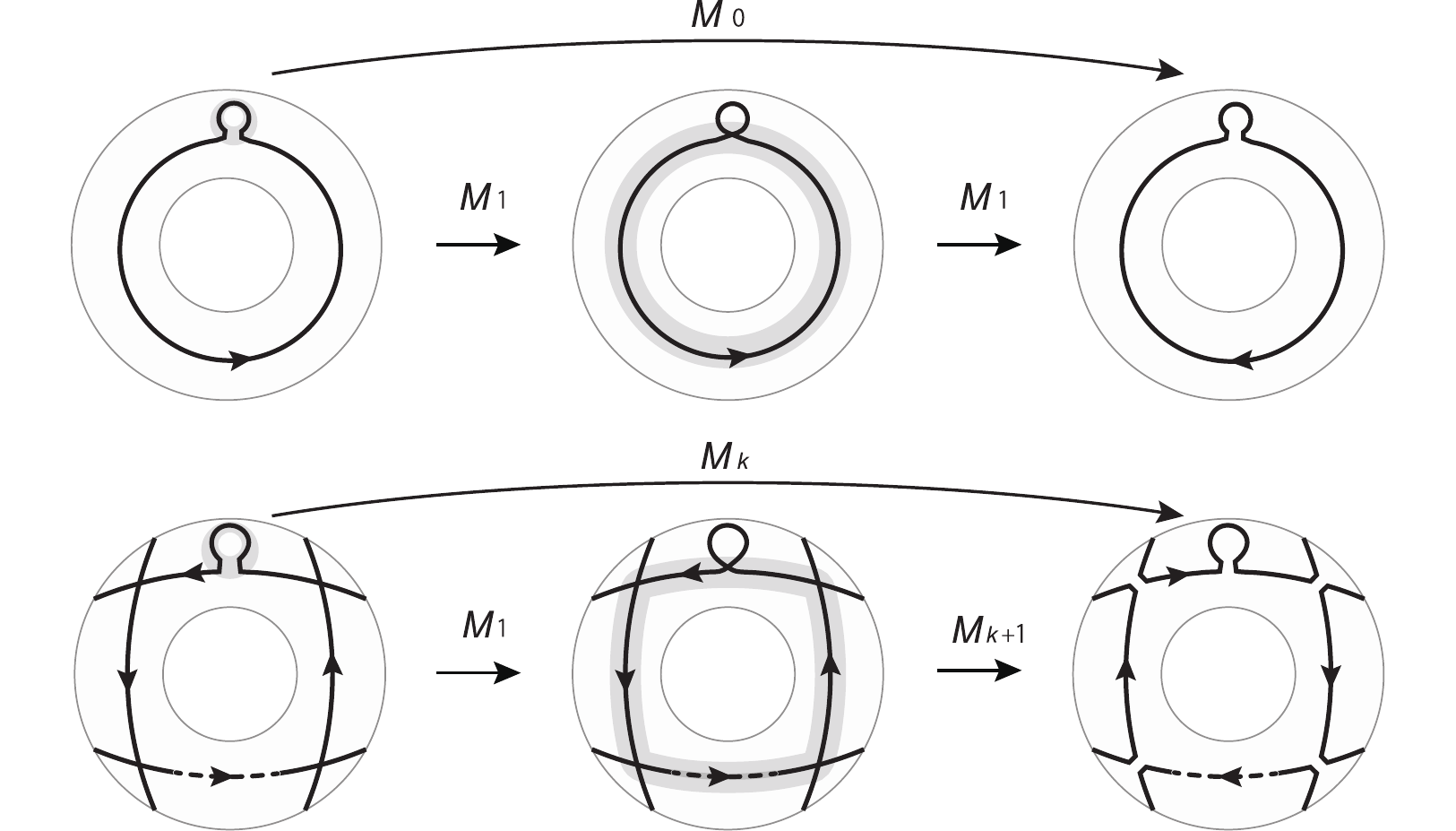}\caption{$M_{k}$ is obtained from $M_1$ and $M_{k+1}$ ($k \ge 0$).}\label{Mdown}\end{figure}

\vskip 2mm
\noindent
{\bf Case 2.} $ i \ge n$.

\noindent
For every integer $k$ with $k \ge 3$, $M_{k+1}$ is obtained from $M_2$, $M_3$ and $M_k$ by using the process shown in Figure \ref{Mup}.

\begin{figure}[h!]\includegraphics[width=12.5cm]{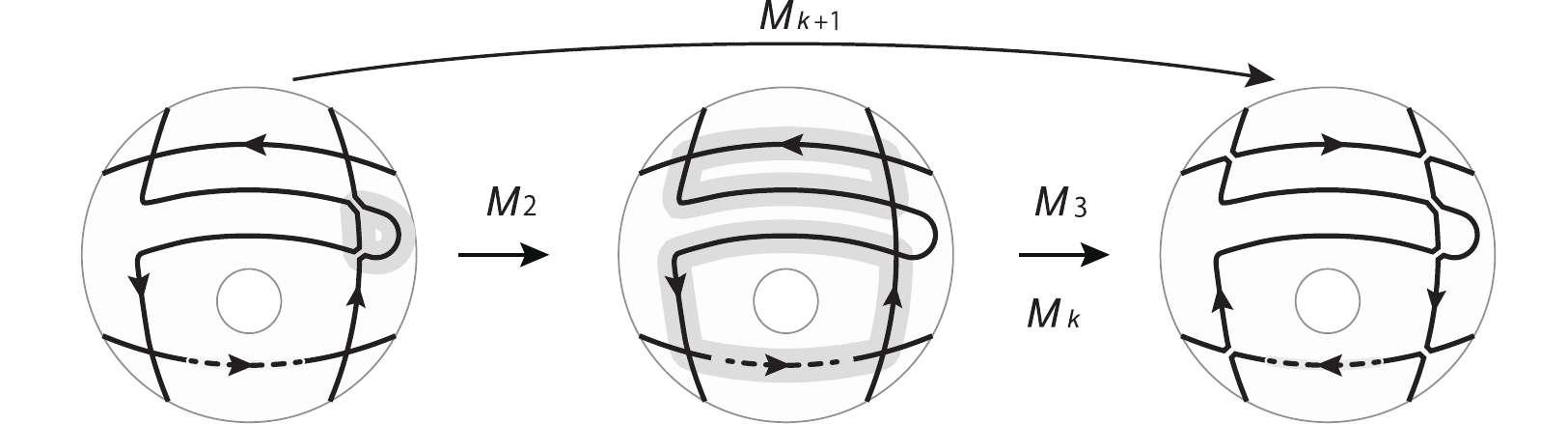}\caption{$M_{k+1}$ is obtained from $M_2$, $M_3$ and $M_{k}$ ($k \ge 3$).}\label{Mup}\end{figure}

Therefore, the claim holds for $i = n, n+1,$ $\ldots$ .\end{proof}
		
		\begin{lemma}[cf.~\cite{ST}]\label{seifertcircle}Any oriented link projection which has at least one crossing point contains a coherent $n$-gon with $n>0$.\end{lemma}
		
		\begin{proof}
Let $P$ be an oriented link projection and let $P'$ be a connected component of $P$ which has at least one crossing point.
We operate smoothing for all crossing points of $P'$ as illustrated in Figure \ref{nosmoothing}. Then, we obtain mutually disjoint circles $C_1, \ldots, C_m$ in $S^2$.
There is a disk $D$ in $S^2$ such that $\partial D$ is one of these circles and ${\rm int}(D) \cap (C_1 \cup \cdots \cup C_m) = \emptyset$.
The oriented simple closed curve corresponding to $\partial D$ is a coherent $n$-gon of $P$ with $n >0$.
		\end{proof}
		
		\begin{proof}[Proof of Theorem \ref{mainthm1}] \ 

\noindent
{\bf (2) $\Rightarrow$ (1)}
Since every $M_i$-move does not change the circle arrangement of an oriented link projection, the claim holds.	
\vskip 2mm
\noindent
{\bf (1) $\Rightarrow$ (2)}
Let $P$ and $Q$ be oriented link projections with $\tau(P)=\tau(Q)$.
Any oriented link projection with at least one crossing point has a coherent $n$-gon with $n > 0$ by Lemma \ref{seifertcircle}.
Then by applying $M_i$ we have a new oriented link projection with fewer crossing points.
By repeating the argument we have an oriented link projection $P'$ (resp.~$Q'$) with no crossing points from $P$ (resp.~$Q$) with
$P \stackrel{\mathbb{Z}_{\ge 1}} {\sim} P'$ (resp. $Q \stackrel{\mathbb{Z}_{\ge 1}} {\sim} Q'$).
Since $\tau(P')=\tau(P)=\tau(Q)=\tau(Q')$ and $\tau(P')$ (resp.~$\tau(Q')$) is obtained from $P'$ (resp. $Q'$) by forgetting the orientation, we see $P' \stackrel{\{ 0 \}} {\sim} Q'$.
Thus we have $P \stackrel{\mathbb{Z}_{\ge 0}} {\sim} Q$.
\vskip 2mm
\noindent
{\bf (3) $\Rightarrow$ (2)} It is clear by the definition.
\vskip 2mm
\noindent
{\bf (2) $\Rightarrow$ (3)} The claim follows from Lemma~\ref{lemmaMi}.
\end{proof}

		
\section{Proof of Theorem~\ref{mainthm2}}\label{sec1}
Before starting the proof of Theorem~\ref{mainthm2}, we prepare some definitions and Lemma~\ref{lemmaoddeven} (for a special case, see \cite{IT}) and Lemma~\ref{lemmagooddiagram}.
The Reidemeister moves RI, RI\!I and RI\!I\!I of oriented link projections are depicted in Figure \ref{reidemeister}.

		\begin{figure}[h!]\includegraphics[width=12.5cm]{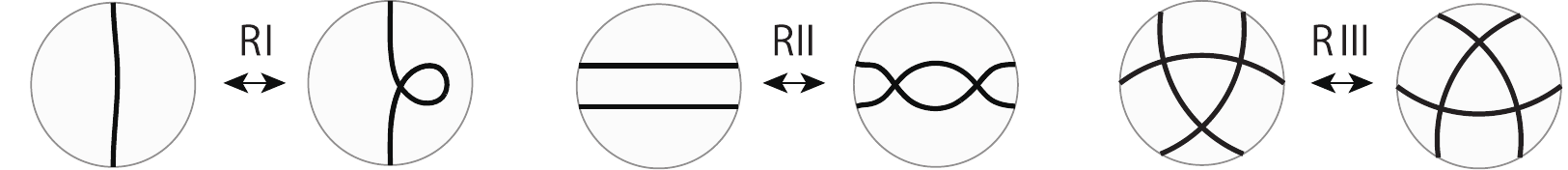}\caption{Reidemeister moves RI, RI\!I and RI\!I\!I with respect to oriented link projections.  Any orientation is allowed.}\label{reidemeister}\end{figure}

Reidemeister moves of an oriented link projection are concerned with $n$-gons ($n=1,$ $2$, $3$). Each Reidemeister move with an arc orientation for either $2$-gon or $3$-gon can be divided into two types as follows. The Reidemeister moves RI\!I and RI\!I\!I are {\it strong} (resp.~{\it weak}) if those have a coherent polygon (resp.~have no coherent polygons) as shown in Figure~\ref{strongweak}.

		\begin{figure}[h!]\includegraphics[width=12.5cm]{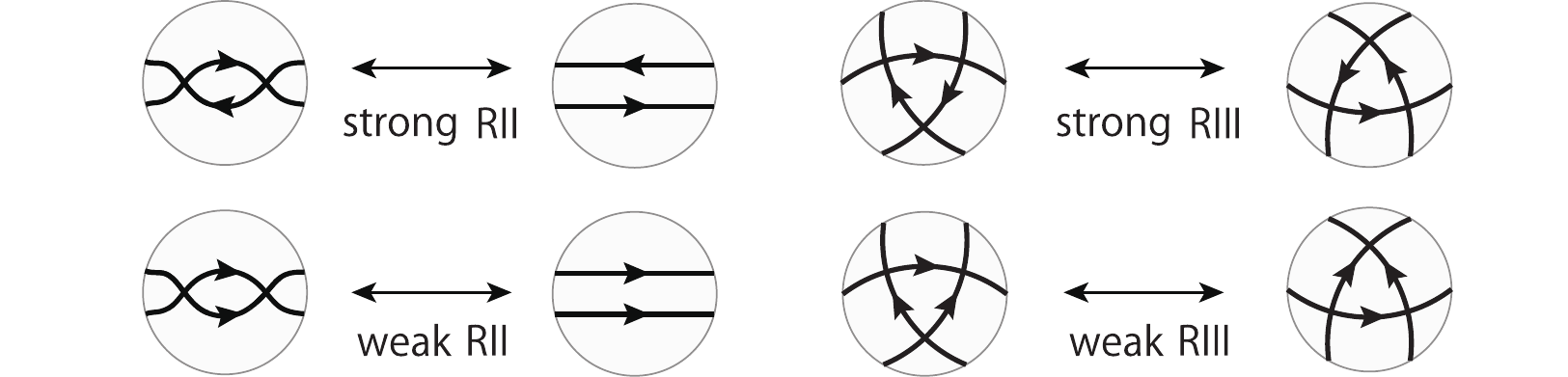}\caption{A strong RI\!I, weak RI\!I, strong RI\!I\!I and weak RI\!I\!I.  }\label{strongweak}\end{figure}
	
For a link projection with no crossing points, we temporarily define a local move $T$-move as illustrated in Figure \ref{Tmove}.
We use $T$-moves only in the following proofs.

\begin{remark}\label{henkaT}{\rm
If link projections $P$ and $P'$ with no crossing points are related by an application of $T$-move, $\# \tau(P)$ differs from $\# \tau(P')$ by $\pm 1$.
}\end{remark}

		\begin{figure}[h!]\includegraphics[width=12.5cm]{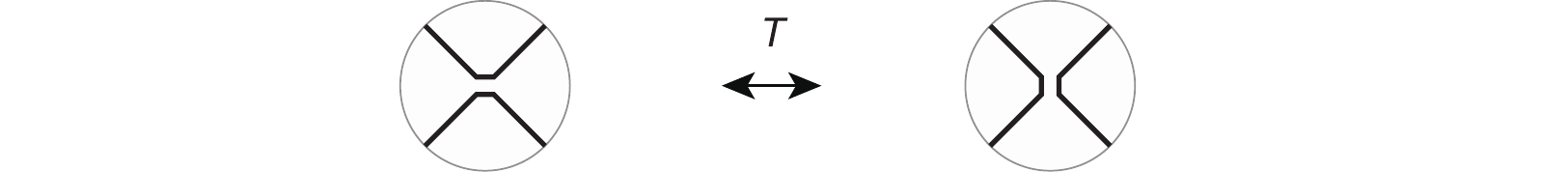}\caption{The $T$-move.}\label{Tmove}\end{figure}
		
\begin{lemma}\label{lemmaoddeven}
Let $D_1$ and $D_2$ be link diagrams and let $D$ be an $n$-component link diagram.
\begin{enumerate}
\item[(1)] If $D_1$ is related to $D_2$ by exactly one Reidemeister move,\\
then $\#\tau(D_1) - \#\tau(D_2)=0$ or $\pm 2$.
\item[(2)] If $n$ is odd $($resp.~even$)$, then $\# \tau(D)$ is odd $($resp.~even$)$.
\end{enumerate}
\end{lemma}

		\begin{proof} \ \\
{\bf (1)} Let $P_i$ be an oriented link projection obtained by forgetting over/under information of $D_i$ ($i=1,2$).
We consider the following two cases.
\vskip 2mm
\noindent
{\bf Case 1.} $P_2$ is obtained from $P_1$ by one RI or one strong RI\!I.\\
Since RI (resp.~strong RI\!I) is $M_1$ (resp.~$M_2$), the move does not change the circle arrangement of $P_2$. Therefore $\# \tau(P_1) - \# \tau(P_2)=0$.
\vskip 2mm
\noindent
{\bf Case 2.} $P_2$ is obtained from $P_1$ by one weak RI\!I or one RI\!I\!I.\\
As illustrated in Figure \ref{weakR2_R3}, $\tau(P_1)$ is obtained from $\tau(P_2)$ by applying exactly two times applications of $T$-moves.
By Remark \ref{henkaT}, $\# \tau(P_1) - \# \tau(P_2)=0$ or $\pm 2$.

		\begin{figure}[h!]\includegraphics[width=12.5cm]{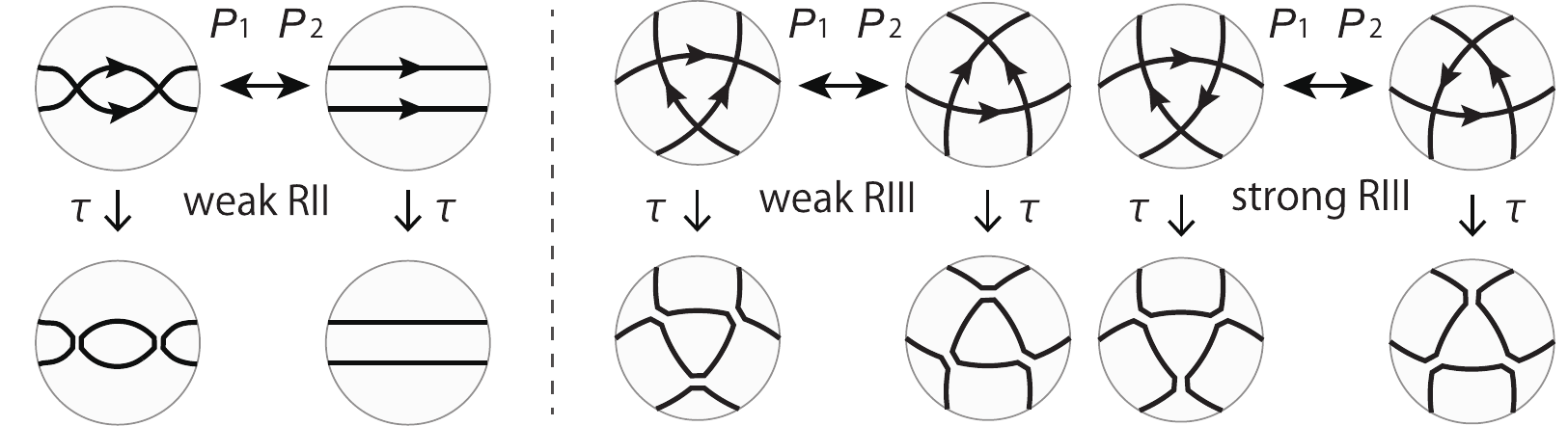}
\caption{A weak RI\!I (resp. RI\!I\!I) changes the number of circles of the circle arrangement by $0$ or $\pm 2$.}\label{weakR2_R3}\end{figure}

By Case 1 and Case 2, $\#\tau(D_1) - \#\tau(D_2)=0$ or $\pm 2$.
\vskip 1mm
\noindent
{\bf (2)} Let $P$ be an oriented link projection obtained by forgetting over/under information of $D$. Since an arbitrary pair of oriented $n$-component link projections is related by Reidemeister moves in Figure~\ref{reidemeister}, $P$ is related to an oriented link projection consisting of $n$ simple closed curves.
By Lemma \ref{lemmaoddeven} (1), $\#\tau(P)$ is odd (resp.~even) if $n$ is odd (resp.~even).\end{proof}
		
\begin{remark}\label{lemma_weak}{\rm
By Case 2 of the proof of Lemma \ref{lemmaoddeven} (1), we understand that the weak RI\!I is divided into the two types. One changes the number of circles of the circle arrangement by exactly 2 and the other does not change it.
}\end{remark}
		
For an oriented link diagram $D$ and a disk $\Delta \subset S^2$, we call $\Delta$ a {\it disk containing three arcs} if 
\begin{enumerate}
\item[$\bullet$] the intersection $\Delta \cap D$ consists of a disjoint union of three simple arcs properly embedded in $D$, 
\item[$\bullet$] these arcs are contained in mutually different components of $\tau(D)$ and
\item[$\bullet$] each arc is an outermost arc of $D$.
\end{enumerate}
In addition, if the arc orientations are as illustrated in Figure \ref{3arcs} or all inverse orientations, we call $\Delta$ a {\it disk containing three coherent arcs}. See Figure \ref{3arcs}.

		\begin{figure}[h!]\includegraphics[width=12.5cm]{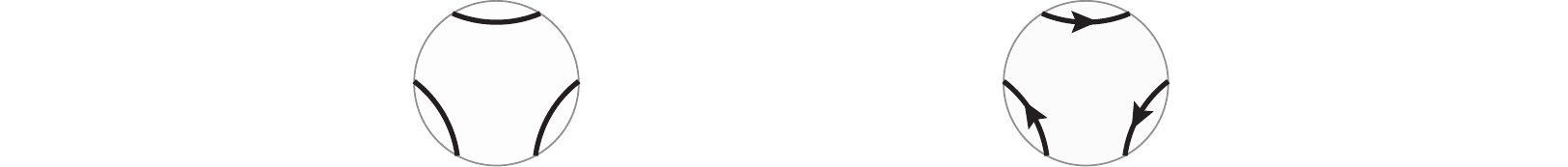}\caption{A disk containing three arcs and a disk containing three coherent arcs. These arcs belong to different components of $\tau(D)$.}\label{3arcs}\end{figure}
		
		\begin{lemma}\label{lemmagooddiagram} Let $D$ be a diagram of an oriented link $L$ with $\#\tau(D) \ge 3$. Then, there is a diagram $D'$ of $L$ such that $D'$ has a disk containing three coherent arcs and $\#\tau(D')=\#\tau(D)$.
		\end{lemma}

		\begin{proof}
We deform $D$ by RI moves and weak RI\!I moves for oriented link diagrams.
\vskip 2mm
\noindent
{\bf Case 1.} There is a disk $\Delta_1$ containing three arcs.\\
After a RI of $D$ in $\Delta_1$ if necessary, we can find out a disk $\Delta$ containing three coherent arcs. See Figure~\ref{type1}.\\
		\begin{figure}[h!]\includegraphics[width=12.5cm]{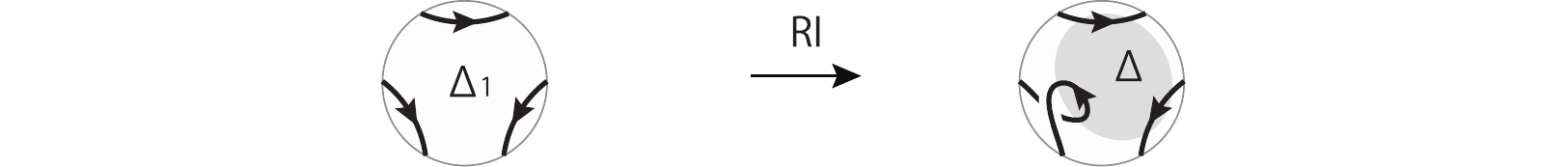}\caption{By using RI if necessary, we obtain a disk $\Delta$ containing three coherent arcs.}\label{type1}\end{figure}

\noindent
{\bf Case 2.} There are no disks containing three arcs.\\
Let $\gamma$ be a simple arc in $S^2$ that intersects three or more components of $\tau(D)$ transversally and away from the crossings of $D$. Then it is easy to see that by shortening $\gamma$ if necessary and by taking a regular neighborhood of $\gamma$, we have a disk $\Delta_2$ with the following properties.
The $\Delta_2$ contains arcs $a$, $b_1$, $\cdots$, $b_k$ and $c$ as illustrated in Figure~\ref{nonoutermost}, where the arcs $b_1$, $\cdots$, $b_k$ ($k \ge 1$) belong to the same circle in $\tau(D)$ and the arcs $a$, $b_i$ and $c$ are contained in mutually different circles in $\tau(D)$.
By using weak RI\!I moves (and RI moves if necessary) on the disk $\Delta_2$, we obtain a diagram of $L$ as illustrated in Figure~\ref{nonoutermost}, where an orientation pattern is shown as an example.

		\begin{figure}[h!]\includegraphics[width=12.5cm]{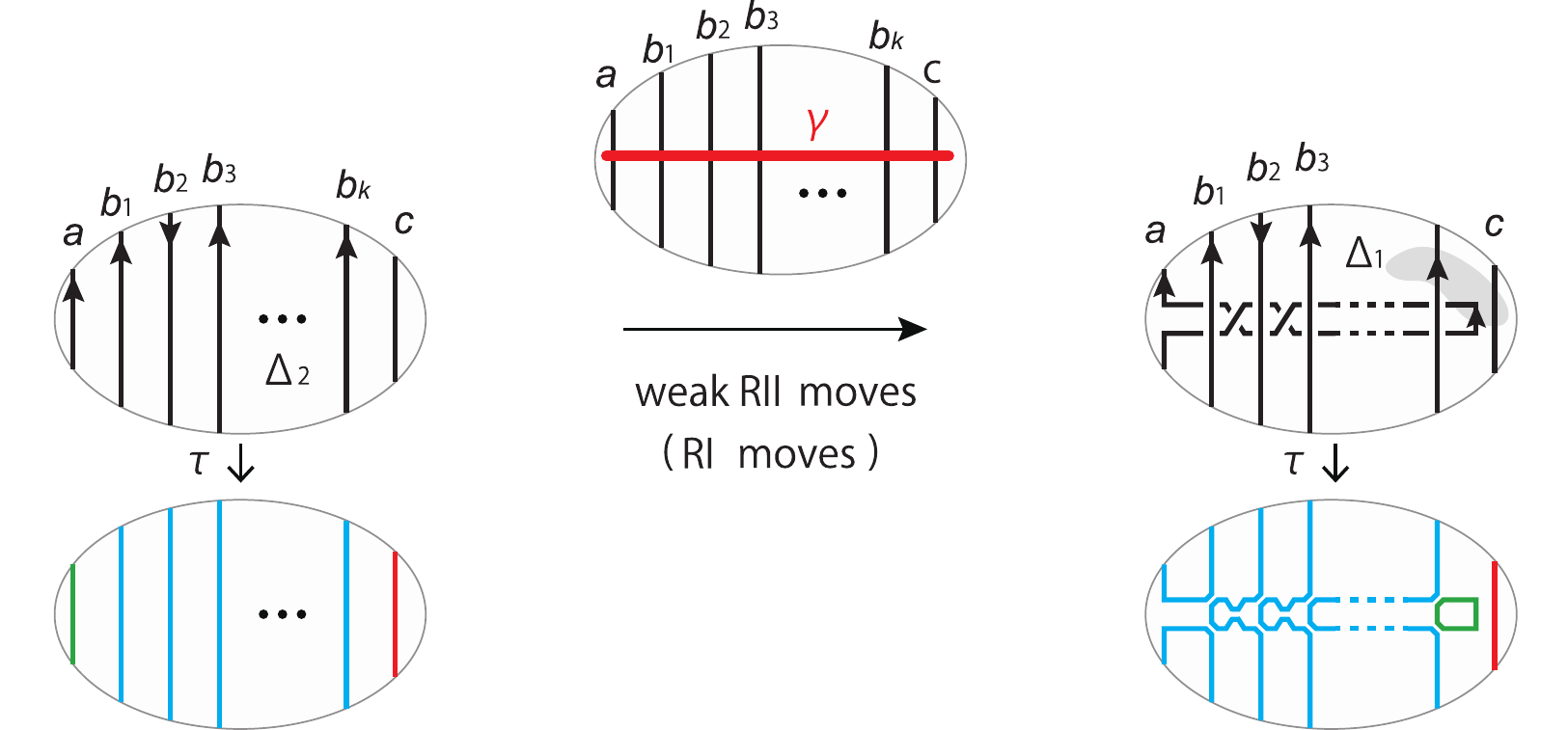}\caption{By using RI moves and weak RI\!I moves  at a disk $\Delta_2$, we have a disk $\Delta_1$ containing three arcs.}\label{nonoutermost}\end{figure}

These weak RI\!I moves do not change the number of circles contained in the circle arrangement.
See Remark \ref{lemma_weak}.
Therefore, we can find out a disk $\Delta_1$ containing three arcs and we have the condition of Case 1.\end{proof}

		\begin{proof}[Proof of Theorem \ref{mainthm2}]
Let $D_0$ be a diagram of an oriented $n$-component link $L$ with $\#\tau(D_0) \ge 3$.
By Lemma~\ref{lemmagooddiagram},
there is a diagram $D_0'$ of $L$ which has a disk $\Delta$ containing three coherent arcs with $\#\tau(D_0)=\#\tau(D_0')$.
We transform $D_0'$ into a diagram $D_1$ of $L$ by the deformation in the disk $\Delta$ as illustrated in Figure~\ref{boro} and we have $\#\tau(D_1)=\#\tau(D_0')-2$.
That is, the deformation changes the number of circles contained in the circle arrangements exactly by 2.

		\begin{figure}[h!]\includegraphics[width=12.5cm]{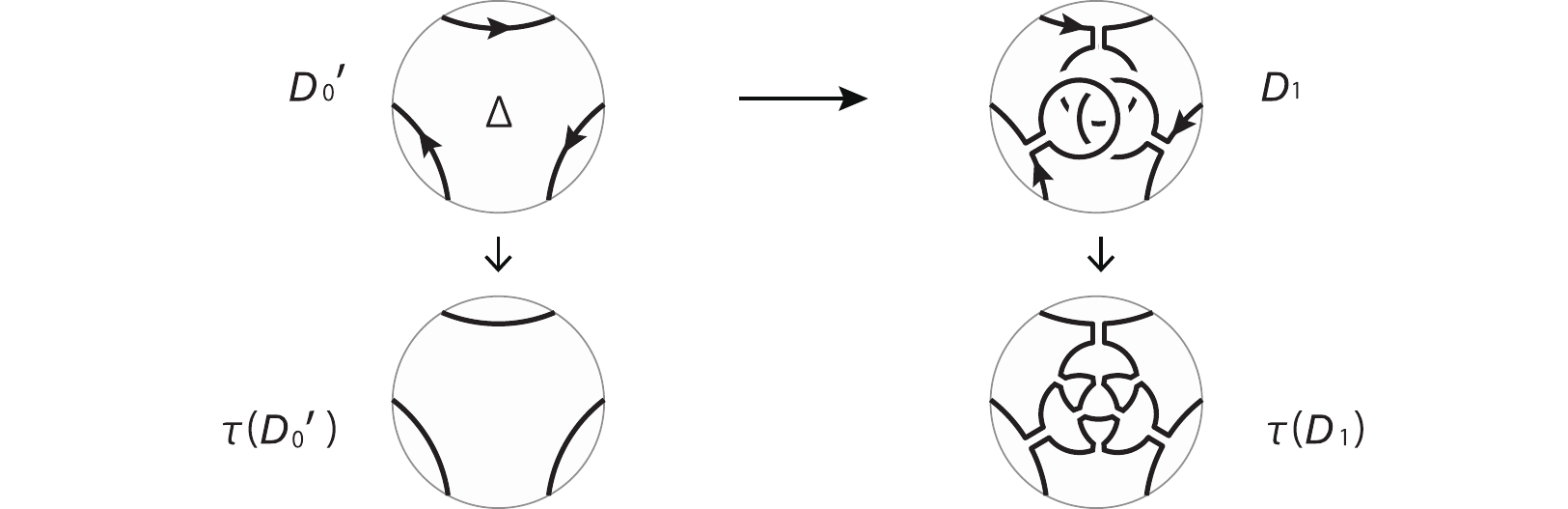}
\caption{A deformation of a diagram in a disk $\Delta$ containing three coherent arcs.}\label{boro}\end{figure}
Therefore, by using the deformations finitely many times, we can obtain a diagram $D$ of $L$ from $D_0$ where $\#\tau(D) = 2$ or 1. By Lemma~\ref{lemmaoddeven} (2), the minimum number of circles of the circle arrangement of every diagram of $L$ is 1 (resp.~2) if $n$ is odd (resp.~even). \end{proof}


\section*{Acknowledgements} 
The first author was supported by a Waseda University Grant for Special Research Projects (Project number:~2014K-6292) and the JSPS Japanese-German Graduate Externship.
He was a project researcher of Grant-in-Aid for Scientific Research (S) 24224002 (2016.4--2017.3).
The third author was partially supported by Grant-in-Aid for Challenging Exploratory Research (No.~15K13439) and Grant-in-Aid for Scientific Research(A) (No.~16H02145), Japan Society for the Promotion of Science.
The authors would like to thank Dr.~Kuniyuki Takaoka for his helpful comments.
The authors are very grateful to the referee for his/her helpful comments.

\end{document}